\documentclass[12pt]{amsart}

\usepackage[text={6.5in,9in},centering]{geometry}

\usepackage{color,hyperref}

\usepackage{graphicx}
\graphicspath{{FigPNG/}}

\usepackage{amssymb,amsmath,amsfonts}
\usepackage{amsthm,aliascnt}
\usepackage{mathrsfs}
\usepackage[config,labelfont={rm,rm},textfont=rm,position=bottom]{caption,subfig}
\usepackage{rotating}

\usepackage{mdwlist}

\usepackage{amssymb}
\usepackage{amsmath}
\usepackage{latexsym}
\usepackage{mathrsfs}
\usepackage{amsthm}
\usepackage{bm}

\newcommand{\Reals}[1]{{\rm I\! R}^{#1}}
\newcommand{\mymu}{\ensuremath{\kappa}}
\newcommand{\alambda}{\ensuremath{\lambda}}

\newtheorem{theorem}{Theorem}[section]

\newtheorem{lemma}[theorem]{Lemma}
\newtheorem{remark}[theorem]{Remark}
\newtheorem{corollary}[theorem]{Corollary}
\theoremstyle{definition}
\newtheorem{algorithm}[theorem]{Algorithm}

\numberwithin{equation}{section}

\def\l1{{$\ell_1$}}
\def\trhotg{\wt{\rho}_{TG}}

\def\GEa{\Gamma_E}
\def\GWe{\Gamma_W}

\newcommand{\lp}{\left(}
\newcommand{\rp}{\right)}

\newcommand{\inv}[1]{#1^{-1}}

\newcommand{\wt}[1]{\widetilde{#1}}

\def\b1{\bm{ 1}}

\def\bu{\bm{ u}}
\def\bw{\bm{ w}}

\def\br{\bm{ r}}

\def\bx{\bm{ x}}

\def\bbf{\bm{ f}}

\newcommand{\A}{{\mathcal A}}

\newcommand{\de}{{:=}}

\newcommand{\vek}[1]{\bm{#1}}
\newcommand{\bv}{{\bm{v}}}

\newcommand{\tl}{{\rm TL}}

\everymath{\displaystyle}
\title[{Polynomial of best uniform approximation to $x^{-1}$ and two-level methods}]
{Polynomial of best uniform approximation to $x^{-1}$ and smoothing in two-level methods}
\author[J.~K.~Kraus]{Johannes K. Kraus}\address{Johann Radon Institute for
  Computational and Applied Mathematics, Austrian Academy of
  Sciences Altenberger Str. 69, 4040 Linz,
  Austria.}\email{johannes.kraus@oeaw.ac.at}
\author[P.~S.~Vassilevski]{Panayot S. Vassilevski} \address{Center for Applied Scientific
  Computing, Lawrence Livermore National Laboratory, P.O. Box 808,
  L-560, Livermore, CA 94550, USA.} \email{panayot@llnl.gov}
\author[L.~T.~Zikatanov]{Ludmil T. Zikatanov}\address{Department of Mathematics, The
  Pennsylvania State University, University Park, PA 16802, USA.} \email{ltz1@psu.edu}
 
\date{\today}

\thanks{
The work of the first author has been supported by the Austrian Science Fund, Grant P22989-N18.
The work of the second author is performed under the auspices of the U.S. Department 
of Energy by Lawrence Livermore National Laboratory under Contract 
DE-AC52-07NA27344. The work of the third author is supported in part
by the National Science Foundation, DMS-0810982, U.S. Department of
Energy (LLNL-B595949) and DoE grant DE-FG02-11ER26062/DE-SC0006903.}
 
\begin{document}

\maketitle

\begin{abstract}
  We derive a three-term recurrence relation for
  computing the polynomial of best approximation in the uniform norm
  to $x^{-1}$ on a finite interval with positive endpoints.  As
  application, we consider two-level methods for scalar elliptic
  partial differential equation (PDE), where the relaxation on the
  fine grid uses the aforementioned polynomial of best
  approximation. Based on a new smoothing property of this polynomial
  smoother that we prove, combined with a proper choice of the coarse
  space, we obtain as a corollary, that the convergence rate of the
  resulting two-level method is uniform with respect to the mesh
  parameters, coarsening ratio and PDE coefficient variation.
 \end{abstract}

\section{Introduction} 
The polynomial of best approximation in uniform norm to $x^{-1}$ on a
finite interval can be found in different forms in many classical
texts on approximation theory, for example,
see~\cite[p.~33,~Equation(4.25)]{1967MeinardusG-aa},
\cite[Exercise~1.20]{Rivlin}. In fact, the approximating polynomial for
$\frac{1}{t-a}$, $a>1$, has already been discovered by Chebyshev in 1887,
see~\cite{2Chebyshev1887}.

As an application, we study two-level methods with smoothers based on
this polynomial of best approximation to $x^{-1}$ on a finite interval
$[\lambda_{\min{}},\lambda_{\max{}}]$,
$0<\lambda_{\min{}}<\lambda_{\max{}}$, in the $\|\cdot\|_{\infty}$
(uniform) norm. We derive several results important for applications:
a three-term recurrence relation for constructing these polynomials;
error estimates; the positivity and monotonicity of the sequence of
polynomials of best approximation, and we use these results in
designing components of two-level methods. We show a major smoothing
property of the polynomial and as a corollary, based on an abstract
two-level estimate we derive two-level (TL or TG) convergence estimates in the
case of discretized elliptic PDE with heterogeneous coefficients. The
estimate explicitly depends on the degree of the polynomial (or on the
range of the spectrum which needs to be resolved by the smoother) and
we prove that if coarse spaces with stability and approximation
properties that are robust with respect coefficient variation are
used, then the two-level methods with polynomial smoothers based on the
polynomial of best approximation to $1/x$ are robust with respect to
the variation in the coefficients of the PDE. Several examples of
coarse spaces that provide the required contrast independent approximation
property are available in the literature, cf., e.g.,
\cite{2010GalvisJ_EfendievY-aa},
\cite{2011ScheichlR_VassilevskiP_ZikatanovL-aa}, and earlier
\cite{1999BrezinaM_HebertonC_MandelJ_VanekP-aa} as modified recently
in \cite{2011BrezinaM_VassilevskiP-aa}).

The paper is organized as follows. In
Section~\ref{sect:three-term-recurrence} we derive a three-term
recurrence relation for the polynomial of best approximation to
$x^{-1}$. Several properties of the sequence of polynomials of best
approximation to $1/x$ are shown in Section~\ref{sect:properties}.  In
Section~\ref{sect:two-level-methods} we discuss and prove the major
smoothing property of the polynomial, which explicitly involves the
polynomial degree and we use it an abstract two-level convergence
result. As a corollary, we derive an estimate for the convergence rate
in case of finite element discretization of scalar elliptic PDE with
coarse spaces that provide contrast independent approximation
resulting in contrast independent two-grid convergence.  This
convergence behavior is illustrated also with numerical tests in
Section~\ref{section:numerical experiments}.

\section{Best polynomial approximation to $x^{-1}$ in uniform
  norm}\label{sect:three-term-recurrence}

We begin with notation and some simple and well known definitions
related to Chebyshev polynomials. We consider a finite interval,
$[\lambda_{\min{}},\lambda_{\max{}}]$, with  
$0<\lambda_{\min{}}<\lambda_{\max{}}<\infty$. 
We denote 
\begin{equation}\label{eqn:define-a}
\kappa=\frac{\lambda_{\max{}}}{\lambda_{\min{}}}, \quad 
\sigma=
\frac{1}{\lambda_{\max{}}-\lambda_{\min{}}}, \quad 
a=\frac{\lambda_{\max{}}+\lambda_{\min{}}}
{\lambda_{\max{}}-\lambda_{\min{}}} 
= \frac{\kappa+1}{\kappa-1}.
\end{equation}
Note that $a>1$ and $\sigma > 0$. The change of variables 
\begin{equation*}
t=
\frac{2}{\lambda_{\max{}}-\lambda_{\min{}}}\left(x-
\frac{\lambda_{\max{}}+\lambda_{\min{}}}{2}\right)=
2\sigma x-a,
\end{equation*}
maps the interval $[\lambda_{\max{}},\lambda_{\min{}}]$ to $[-1,1]$. The inverse map is
\begin{equation*}
x=\frac{1}{2\sigma}(t+a), \quad \mbox{and}\quad  
\frac{1}{x}=\frac{2\sigma}{t+a}.
\end{equation*}
We thus aim to find the polynomial of degree less than or equal to $m$ of best approximation
in the norm $\|\cdot\|_{\infty,[-1,1]}$ of $f(t)=\frac{1}{t+a}$, $a>1$. We note that if $Q_m(t)$
is the polynomial of best approximation to $1/(t+a)$ on $[-1,1]$, and the error of approximation is
\begin{equation*}
E_{[-1,1]}= \min_{Q\in
 \mathcal{P}_m}\left\|\frac{1}{t+a}-Q\right\|_{L_{\infty}[-1,1]},
\end{equation*}
then
 \begin{equation}
  \label{eqn:et}
q_m(x):={2\sigma}Q_m(2\sigma x-a),\quad\mbox{and}\quad
  E=\min_{q\in \mathcal{P}_m}
\left\|\frac{1}{x}-q\right\|_{L_{\infty}[\lambda_{\max{}},\lambda_{\min{}}]}=
   2\sigma E_{[-1,1]}
 \end{equation}
are the polynomial of best approximation in $L^{\infty}$-norm on $[\lambda_{\min{}},\lambda_{\max{}}]$
and the error of approximation, respectively.

We denote the (first kind) Chebyshev polynomial of degree $k$ by $T_k$. For $T_k(\xi)\in \mathcal{P}_k$
we have
\begin{equation*}
T_k(\xi) = 
\frac12\left[(\xi+\sqrt{\xi^2-1})^k+
(\xi+\sqrt{\xi^2-1})^{-k}\right] = 
\frac12\left[(\xi+\sqrt{\xi^2-1})^k+
(\xi-\sqrt{\xi^2-1})^{k}\right].
\end{equation*}
We recall that 
\begin{equation*}
T_{k}(t)=\cos k\arccos(t), \quad t\in [-1,1]
\end{equation*}
and denote 
\begin{equation}\label{eqn:id0}
\delta := a-\sqrt{a^2-1}=
\frac{\sqrt{\kappa}-1}{\sqrt{\kappa}+1}, \quad \eta=-\delta .
\end{equation}
Evidently, $0\le \delta<1$, $\delta^{-1} = a+\sqrt{a^2-1}$, $\eta<0$ and $\delta=|\eta|$.

With this notation in hand, we have the following identities, 
\begin{equation}\label{eqn:id1}
a=-\frac12(\eta+\eta^{-1}), \quad 
\frac{1}{t+a} = \frac{2}{2t-\eta-\eta^{-1}},
\end{equation}
and directly from the expression for $T_k(\xi)$ given above, we also
have
\begin{equation}\label{eqn:id2}
T_k(a) = \tfrac12(-1)^k(\eta^{k}+\eta^{-k}), \quad
T_k(-a) = \tfrac12(\eta^{k}+\eta^{-k}).
\end{equation}

\subsection{Approximation error and three-term recurrence}

Next, in Theorem~\ref{thm:polynomial} we give a representation of the
best polynomial approximation to $\frac{1}{t+a}$ in the $L^{\infty}$-norm
on the interval $[-1,1]$. The proof of this theorem is given in the appendix,
and amounts to showing that the form given in~\eqref{eqn:bpa1x} is equivalent
to the one given in~\cite[p.~33,~Equation~(4.25)]{1967MeinardusG-aa}.
\begin{theorem}\label{thm:polynomial}
  Let $m\ge 1$ be a fixed integer. The polynomial $Q_m\in
  \mathcal{P}_m$, which furnishes the best approximation to
  $\frac{1}{t+a}$ in the $L^{\infty}$-norm on~$[-1,1]$ is
\begin{equation}\label{eqn:bpa1x}
Q_m(t) = \frac{1}{t+a}\left(1-\frac{2 \eta^m}{(\eta-\eta^{-1})^2}R_{m+1}(t)\right),
\end{equation}
where
\begin{equation}\label{eqn:rmplus1}
R_{m+1}(t) = \eta^{-1} T_{m+1}(t)-2 T_m(t)+\eta T_{m-1}(t).
\end{equation}
The error of best approximation is 
\begin{equation*}
E_{[-1,1]} = 
\min_{Q\in \mathcal{P}_m}\left\|\frac{1}{t+a}-Q\right\|_{L_\infty[-1,1]} = 
\frac{\delta^{m}}{a^2-1}.
\end{equation*}
\end{theorem}
\begin{proof} 
We prove this theorem in the appendix by showing how one can
derive~\eqref{eqn:bpa1x} from \cite[p.~33,~Equation~(4.25)]{1967MeinardusG-aa}. 
\end{proof}
The following corollary is immediate and follows after elementary calculations.
\begin{corollary}\label{cor:CGetc} 
Let $E_{m,[\lambda_{\min{}} ,\lambda_{\max{}} ]}$ 
be the error of approximation with polynomial of degree $m$ 
on the interval $[\lambda_{\min{}} ,\lambda_{\max{}} ]$, $0<\lambda_{\min{}} <\lambda_{\max{}} <\infty$. Then
\begin{equation}\label{eqn: error on interval}
E_{m,[\lambda_{\min{}} ,\lambda_{\max{}} ]}=
2\delta^{m-1}
E^2_{0,[\sqrt{\lambda_{\min{}} },\sqrt{\lambda_{\max{}} }]}, 
\end{equation}
where $E_{0,[\sqrt{\lambda_{\min{}} },\sqrt{\lambda_{\max{}} }]}$ is
given by the expression 
\begin{equation*}
E_{0,[\sqrt{\lambda_{\min{}} },\sqrt{\lambda_{\max{}} }]}=
\frac12\left(\frac{1}{\sqrt{\lambda_{\min{}} }}
-\frac{1}{\sqrt{\lambda_{\max{}} }}\right).
\end{equation*}
\end{corollary}
\begin{theorem}\label{thm:CFrec}
  For the polynomials of best approximation to $\frac1x$ given
  in~\eqref{eqn:bpa1x}, the following three-term recurrence relation
  holds:
\begin{equation}\label{eqn:recurrence}
\eta^{-1} Q_{m+2}(t)-2t Q_{m+1}(t)+\eta Q_{m}(t) = -2,\quad m=0,1,\ldots
\end{equation}
with
\begin{equation*}
Q_0(t) = \frac{a}{a^2-1},\qquad
Q_1(t) = \frac{1}{\sqrt{a^2-1}}-\frac{t}{a^2-1}.
\end{equation*}
The error of approximation is:
\begin{equation*}
E_{[-1,1]} = 
\min_{Q\in \mathcal{P}_m}\left\|\frac{1}{t+a}-Q\right\|_{L_\infty[-1,1]} = 
\frac{\delta^{m}}{a^2-1}.
\end{equation*}
\end{theorem}
\begin{proof}
It is immediate to check that for $m=0$,
\begin{equation*}
Q_0(t)=\frac{1}{2}\left(\frac{1}{a-1}+\frac{1}{a+1}\right) =
\frac{1}{2}\left(\frac{2a}{a^2-1}\right)
=\frac{2(\eta+\eta^{-1})}{(\eta-\eta^{-1})^2}. 
\end{equation*}
Setting 
\begin{equation*}
r_m(t)=(\eta-\eta^{-1})^2+q_m(t),\quad\mbox{with}
\quad q_m(t) = 2(-1)^m\eta^{-m}R_{m+1}(t)
\end{equation*} 
we have
\begin{equation*}
Q_m(t) = \frac{r_m(t)}{(t+a)(\eta-\eta^{-1})^2}. 
\end{equation*}
For  $m=1$ we then readily obtain
\begin{eqnarray*}
r_1(t) &=& 
(\eta-\eta^{-1})^2-2\eta^{-1}(\eta(2t^2-1)+2t+\eta^{-1})\\
&=&
\eta^2+\eta^{-2}-2
-4t^2+2-4\eta^{-1}t-2\eta^{-2}\\
&=&\eta^2-4t^2-4\eta^{-1}t-\eta^{-2}=\eta^2-(2t+\eta^{-1})^2\\
& =& 
(\eta-2t-\eta^{-1})(\eta+2t+\eta^{-1})  
=2(\eta-2t-\eta^{-1})(t+a).
\end{eqnarray*}
This shows that $Q_1(t)$ has the form given in the statement of the
theorem.
For $m\ge 2$, using the recurrence relation for $T_m(t)$, it is easy
to check that
\begin{equation*}
R_{m+2}(t)-2tR_{m+1}(t)+R_{m}(t)=0.
\end{equation*}
We then have 
\begin{eqnarray*}
\eta q_{m+1}(t)+2tq_{m}(t) +\eta^{-1} q_{m-1}(t)&=&
2\eta(-1)^{m+1}\eta^{-m-1}R_{m+2}(t)\\
&&~~+4t(-1)^{m}\eta^{-m}R_{m+1}(t)
+2\eta^{-1}(-1)^{m-1}\eta^{-m+1}R_{m}(t)\\
&=&
2(-1)^{m+1}\eta^{-m}(R_{m+2}(t)-2tR_{m+1}(t)+R_{m}(t))=0 . \\
\end{eqnarray*}
On the other hand, for any constant $K$, by the definition of $\eta$ 
we have 
\begin{equation*}
\eta K+2tK+\eta^{-1}K=2(t+a)K. 
\end{equation*}
Hence, after applying the above identities (with $K=(\eta-\eta^{-1})^2$)
we get 
\begin{equation*}
\eta r_{m+1}(t)+2tr_{m}(t) +\eta^{-1} r_{m-1}(t)=
2(t+a)(\eta-\eta^{-1})^2.
\end{equation*}
The proof then is easily completed by using the definition of
$Q_m(t)$.
\end{proof}

The next lemma gives an estimate on $|R_{m+1}(t)|$ by a linear polynomial,
which is used later to derive a sufficient condition for the positivity
of $q_m(\cdot)$.
\begin{lemma}\label{lem:Qmp1_bounds}
The following estimate holds for the polynomial $R_{m+1} (t)$ defined in Theorem~\ref{thm:polynomial}:
\begin{equation}\label{Qmp1_bounds}
-2 (t+a) \le R_{m+1} (t) \le 2 (t+a) , \quad t \in [-1,1].
\end{equation}
\end{lemma}
\begin{proof}
  Recall that by the definition of $\eta$ and $\delta$
  (see~\eqref{eqn:id0}), we have that $\eta<0$, and $|\eta| = \delta$.
  Let $t=\cos\alpha$, for $\alpha\in [0,\pi]$. Then we find that
\begin{eqnarray}\label{Qmp1_lb}
R_{m+1}(t) + 2t-\eta-\eta^{-1}&=&
\eta^{-1}(T_{m+1}(t)-1)-2(T_m(t)-t)+ \eta(T_{m-1}(t)+1) \nonumber \\
&=&
  -2\eta^{-1}\sin^2\frac{m+1}{2}\alpha
+ 4\sin\frac{m+1}{2}\alpha
    \sin\frac{m-1}{2}\alpha \nonumber \\
&&~~- 2\eta\sin^2\frac{m-1}{2}\alpha \nonumber \\
&=&
-2\eta^{-1}\left(
\sin\frac{m+1}{2}\alpha-\eta\sin\frac{m-1}{2}\alpha \right)^2 \nonumber \\
&=&
2\delta^{-1}\left(
\sin\frac{m+1}{2}\alpha+\delta\sin\frac{m-1}{2}\alpha \right)^2 \ge 0.
\end{eqnarray}
In an analogous fashion we obtain
\begin{eqnarray}\label{Qmp1_ub}
R_{m+1}(t) - 2t+\eta+\eta^{-1}
&=& \eta^{-1}(T_{m+1}(t)+1)-2(T_m(t)+t)+ \eta(T_{m-1}(t)+1) \nonumber \\
&=&
  2\eta^{-1}\cos^2\frac{m+1}{2}\alpha
- 4\cos\frac{m+1}{2}\alpha
    \cos\frac{m-1}{2}\alpha \nonumber \\
&&~~ + 2\eta\cos^2\frac{m-1}{2}\alpha \nonumber \\
&=&
2\eta^{-1}\left(
\cos\frac{m+1}{2}\alpha-\eta\cos\frac{m-1}{2}\alpha \right)^2 \nonumber \\
&=&
-2\delta^{-1}\left(
\cos\frac{m+1}{2}\alpha+\delta\cos\frac{m-1}{2}\alpha \right)^2 \le 0. 
\end{eqnarray} 
Combining (\ref{Qmp1_lb}) and (\ref{Qmp1_ub}) and using $2t-\eta-\eta^{-1}=2 (t+a)$
yields the desired result.
\end{proof}

\subsection{Algorithm for finding the polynomial of best uniform approximation to $x^{-1}$\label{sect:algorithm}}
The result in Theorem~\ref{thm:CFrec} gives us the polynomial
approximation on the interval $[\lambda_{\max{}},\lambda_{\min{}}]$. Indeed, the
recurrence relation for $q_{m+1}(x)=2\sigma Q_{m+1}(2\sigma x-a)$ is:
\begin{equation*}
Q_{m+1}(2\sigma x-a) = \eta [-2+2(2\sigma x -a) Q_m(2\sigma x-a) - \eta Q_{m-1}(2\sigma x-a)].
\end{equation*}
Multiplying by $2\sigma$ then gives
\begin{equation*}
q_{m+1}(x) = \eta[-4\sigma + 2\sigma(2\sigma x -a) Q_m(2\sigma x-a)-2\sigma\eta Q_{m-1}(2\sigma x -a)].
\end{equation*}

Based on this identity, we have the following algorithm in which the formulas are
obtained by writing $\eta$, $\sigma$ and $a$ in terms of
$\mu_0=1/\lambda_{\max{}}$ and $\mu_1=1/\lambda_{\min{}}$ and $\delta$
 (defined in~\eqref{eqn:id0}). The reason for choosing these parameters
is because the  constants in the algorithm are symmetric with respect
to $\mu_0$ and $\mu_1$. 

\begin{algorithm}\label{alg:computeBA} 
Set $\mu_0=1/\lambda_{\max{}}$ and $\mu_1=1/\lambda_{\min{}}$.

\begin{enumerate}
\item Calculate the $0$-th order polynomial $q_0$ and the first order
  polynomial $q_1$: 
\begin{equation*}
q_0(x) = \frac12(\mu_0+\mu_1), \quad\mbox{and}\quad 
q_1(x) = \frac12(\sqrt{\mu_0}+\sqrt{\mu_1})^2-\mu_0\mu_1 x.
\end{equation*}
\item For $k=1,\ldots,m-1$, $q_{k+1}$ written as a correction to $q_k$
is computed as follows:
\begin{eqnarray*}
\ell_{k+1}(x) &=& \frac{4\mu_0\mu_1}{(\sqrt{\mu_0}+\sqrt{\mu_1})^2}
[1-q_k(x)\, x] + \delta^2  [q_{k}(x)-q_{k-1}(x)]\\ 
& = & \frac{4\mu_0\mu_1}{(\sqrt{\mu_0}+\sqrt{\mu_1})^2}[1-q_k(x)\, x] + \delta^2 \ell_k(x) , \\
q_{k+1}(x) &=& q_{k}(x)+\ell_{k+1}(x).
\end{eqnarray*}
\end{enumerate}
\end{algorithm}

In other words, we have the relation
\begin{equation}\label{defect correction relation}
q_{k+1}(x)-q_k(x) = \delta^2 (q_k(x)-q_{k-1}(x))+
\frac{4\mu_0\mu_1}{(\sqrt{\mu_0}+\sqrt{\mu_1})^2}\left [ 1 -xq_k(x) \right ].
\end{equation}
This formula can be used to perform stationary iterations towards
solving $A \bu = \bbf$ for a given symmetric and positive definite
matrix $A$ and a given symmetric positive definite preconditioner $D$
to $A$. A standard
stationary iterative method has the form: Given an approximation
$\bv$ to the solution $\bu$ of the linear system in hand, the next
approximation $\bw$ is defined as  
\begin{equation*}
\bw = \bv + R \left (\bbf - A \bv \right ). 
\end{equation*}
A sequence of such approximations, approaching $\bu$ (when the method
is convergent) is obtained by applying this iteration
with $\bw=\bu_{j+1}$, $\bv=\bu_{j}$, $j=0,\ldots$, and, with $\bu_0$,
a given initial guess. 

We now define
\begin{equation*}
R = q_m(D^{-1}A)D^{-1},
\end{equation*}
where $q_m$ is the polynomial of best approximation to $x^{-1}$ on the
interval $\left[ \frac{\lambda}{\kappa},\; \lambda \right ]$ with
$\lambda$ an upper bound for the largest eigenvalue of $D^{-1}A$ and
$\kappa>1$, a parameter controlling the length of the interval.

At every iteration, we need to compute the actions $R \br$, where $\br
= \bbf - A \bv$ is the current residual. This is accomplished by
writing 
equation~\eqref{defect correction  relation} with a matrix argument, namely:
\begin{equation}\label{defect correction relation for matrices}
\begin{array}{rcl}
\ell_k(D^{-1}A)&=&q_{k}(D^{-1}A)-q_{k-1}(D^{-1}A),\\
\ell_{k+1}(D^{-1}A)D^{-1} &=& \delta^2 \ell_k(D^{-1}A)D^{-1}\\
&&\quad +
\frac{4\mu_0\mu_1}{(\sqrt{\mu_0}+\sqrt{\mu_1})^2}D^{-1}\left [ I  -Aq_k(D^{-1}A)D^{-1} \right ].\\
\end{array}
\end{equation}
\begin{algorithm}[Polynomial Preconditioning with $R = q_m(D^{-1}A)D^{-1}$]
\label{algorithm: polynomial preconditioning}
\hfill

Given $\br$, in the following steps the algorithms computes at the end
$ q_m(D^{-1}A)D^{-1} \br$.
\begin{itemize}
\item [(0)]  
Initially, compute ${\overline \br} = D^{-1}\br$. 

\item [(i)] Then,  compute
$\bv_0 = \frac{1}{2}\;(\mu_0+\mu_1) {\overline \br}$ and
$\bv_1 = \frac{1}{2} \left (\sqrt{\mu_0}+\sqrt{\mu_1}\right )^2
{\overline \br} - \mu_0\mu_1 D^{-1}A {\overline \br}$.

\item [(ii)] For
$k=1,2,\;\dots,\;m-1$, compute the current and preconditioned
residuals,
\begin{equation*}
\br_k = \br - A \bv_k, \qquad {\overline \br_k} = D^{-1} \br_k.
\end{equation*} 
The next $\bv_{k+1}$ is computed based on the recurrence formula
\eqref{defect correction relation for matrices}
\begin{equation*}
  \bv_{k+1} = \bv_k +\delta^2(\bv_k-\bv_{k-1})+\frac{4\mu_0\mu_1}{(\sqrt{\mu_0}+\sqrt{\mu_1})^2}{\overline \br}_k. 
\end{equation*}
\item [(iii)]
At the end, we let $R \br = \bv_m$. 
\end{itemize}
\end{algorithm}
The reason to write $q_{k+1}$ as a correction to $q_k$ is to show that such iterations look like
iterations in a defect-correction method: First computing the residual $[1-q_{k}(x)\, x]$, and then
trying to correct it by adding an additional term. One can also easily see that for any initial 
$q_0$ and $q_1$, if the sequence $q_k(x)$ converges, then it converges to $x^{-1}$.
In other words, choosing $q_0$ and $q_1$  different from what they are above, will not generate
the sequence of best approximations to $x^{-1}$, but still this sequence will converge to $x^{-1}$.


\section{Properties of the sequence of polynomials\label{sect:properties}}

To simplify the presentation, we now set $\lambda=\lambda_{\max{}}$
and in this notation we have $\lambda_{\min{}}=\frac{\alambda}{\mymu}$
(recall the definition of $\kappa$ given
in~\S\ref{sect:three-term-recurrence}). We thus consider the best
approximation $q_m(x)$ to $\frac1x$ on the interval
$\left[\frac{\alambda }{\mymu},\alambda \right]$.  We prove several
results on the positivity of the polynomial $q_m(x)$, and the
monotonicity of the sequence $\{q_m\}$ for sufficiently large $m$.

We first note the following identity
\begin{equation}\label{eqn:xqx}
\begin{array}{rcl}
x \, q_{m}(x) &=&2 \sigma x \, Q_m(2 \sigma x - a) =(t+a) \, Q_m (t)\\
 &=& 1 - \frac{2 \eta^{m}}{(\eta-\eta^{-1})^2} R_{m+1} (t)
= 1 - \frac{2 (-1)^m\delta^{m}}{(\delta-\delta^{-1})^2} R_{m+1} (t) , \quad t
\in [-1,1]
\end{array}
\end{equation}
This gives 
\begin{equation}\label{eqn:one-minus-xqx}
1-xq_m(x)=\frac{2 (-1)^m\delta^{m}}{(\delta-\delta^{-1})^2} R_{m+1} (t).
\end{equation} 
The next Lemma shows that $(1-q_m(x) x) > 0$ for all $x \in
\left[0,\frac{\alambda }{\mymu}\right]$.
\begin{lemma}\label{lemma:one-minus-xqx}
Let $q_m(x)$ be the polynomial of degree less than or equal to $m$,
which furnishes the best approximation to $\frac1x$ in the $L^{\infty}$-norm
on the interval $\left[\frac{\alambda }{\mymu},\alambda \right]$, $\mymu > 1$.
Then the following inequality holds:
\begin{equation}\label{eq001}
0 < 1-xq_m(x), \quad \forall x \in\left(0,{\frac{\alambda }{\mymu}}\right]
\end{equation}
\end{lemma}
\begin{proof}
Consider the polynomial
\begin{equation*}
p(x) = 1-xq_m(x).
\end{equation*}
Note that $p(x)$ is of degree at most $(m+1)$. Since we have 
\begin{equation*}
p(x) = x\left(\frac1x-q_m(x)\right),
\end{equation*}
and $x > 0$ in the intervals of interest, we may conclude that the
sign changes in the function $\left(\frac1x-q_m(x)\right)$ are the same as the
sign changes in $p(x)$ for any $x>0$. However, $q_m(x)$ is the
polynomial of best uniform approximation to $\frac1x$, and hence there
are at least $(m+2)$ points of Chebyshev alternance in the interval
$\left[\frac{\alambda }{\mymu},\alambda \right]$. Thus, there exist points
$\{x_k\}_{k=1}^{m+2}$ such that
\begin{equation*}
\frac{\alambda }{\mymu} \le x_1< x_2 <\ldots< x_{m+1}<x_{m+2}\le \alambda ,
\end{equation*}
and also such that 
\begin{equation*}
\left(\frac1{x_k}-q_m(x_k)\right) = -\left(\frac1{x_{k+1}}-q_m(x_{k+1})\right), \quad k=1,\ldots,(m+1).
\end{equation*}
We define now $e:=\left(\frac1{x_1}-q_m(x_1)\right)$, and use the alternation property to get that 
\begin{equation*}
p(x_k)p(x_{k+1})=-x_kx_{k+1}e^2<0, \quad k=1,\ldots,(m+1) .
\end{equation*}
Hence, we may conclude that all the roots of $p(x)$ are disjoint,
and that each of them lies in the \emph{open} interval $(x_k,x_{k+1})$,
$k=1,\ldots,(m+1)$.  We may also conclude that there are no roots of
$p(x)$ outside of the open interval
$\left(\frac{\alambda }{\mymu},\alambda \right)$ and there are
no roots of its first derivative outside this interval. This is so by
the Rolle's theorem: the first derivative $p^\prime(x)$ clearly has
$m$ distinct roots, each lying between the roots of $p(x)$. Hence,
$p(x)$ is either strictly increasing or strictly decreasing on the
interval $\left[0,{\frac{\alambda }{\mymu}}\right]$ and also it cannot have a zero in this
interval.  Recall that $0<\delta = -\eta < 1$ and that
$T_k(-1)=(-1)^k$.  Using the definition of $R_{m+1}(t)$
from Theorem~\ref{thm:polynomial}, and the relation~\eqref{eqn:one-minus-xqx} it follows that
\begin{eqnarray*}
p\left( {\frac{\alambda }{\mymu}} \right)&=&\frac{2 (-1)^m\delta^{m}}{(\delta-\delta^{-1})^2}
R_{m+1}(-1)\\
& = &
\frac{2 (-1)^m\delta^{m}}{(\delta-\delta^{-1})^2}
[(-\delta^{-1})(-1)^{m+1} - 2(-1)^{m}+(-\delta)(-1)^{m-1}]\\
&=& \frac{2\delta^{m}}{(\delta-\delta^{-1})^2}
(\delta^{-1}+\delta - 2) =
\frac{2\delta^{m}}{(\delta+\delta^{-1}+2)}<1=p(0).
\end{eqnarray*}
Here we have used that 
\begin{equation}\label{eqn:delta-delta1}
(\delta-\delta^{-1})^2=
[(\delta^{\frac12}+\delta^{-\frac12})^2
(\delta^{\frac12}+\delta^{-\frac12})^2
= (\delta+\delta^{-1}-2) (\delta+\delta^{-1}+2).
\end{equation}
We thus conclude that $p(0) > p(\frac{\alambda }{\mymu})$ and
therefore $p(x)$ must be decreasing on 
$\left(0,{\frac{\alambda}{\mymu}}\right]$, and this leads to
\begin{equation}\label{eqn:002}
0<
\frac{2\delta^{m}}{(\delta+\delta^{-1}+2)}=
 p\left({\frac{\alambda}{\mymu}}\right)
\le p(x)\le 1 ,
\end{equation}
which concludes the proof.
\end{proof}
The next lemma shows that for $x\in\left[0,{\frac{\alambda }{\mymu}}\right]$ the sequence of
polynomials of best approximation of increasing degree is monotone.
\begin{lemma}\label{lem:monotonicity}
The following estimate holds:
\begin{equation}\label{eq003}
q_m(x)<q_{m+1}(x), \quad \mbox{for all}\quad x \in\left[0,{\frac{\alambda }{\mymu}}\right],
\end{equation}
where $q_k(x)$, $k=m,(m+1)$ is the best polynomial approximation of degree at most $k$ to $\frac1x$
in the $L^{\infty}$-norm on the interval $\left[\frac{\alambda }{\mymu},\alambda \right]$,
$\mymu > 1$.
\end{lemma}
\begin{proof}
The proof  amounts to showing that $\ell_{m+1}(x)>0$ (defined in Step 2. of Algorithm~\ref{alg:computeBA})
for $x\in\left[0,{\frac{\alambda }{\mymu}}\right]$. With the notation
given in Algorithm~\ref{alg:computeBA} for such values of $x$ we have 
$x\leq {\frac{\alambda }{\mymu}}=\mu_1^{-1}$.  Therefore, 
\begin{eqnarray*}
\ell_1(x) &=& q_1(x)-q_0(x) = 
\frac12(\mu_0+\mu_1+2\sqrt{\mu_0\mu_1})-\mu_0\mu_1 x - 
\frac12(\mu_0+\mu_1) \\
& = &
\sqrt{\mu_0\mu_1}(1-x\sqrt{\mu_0\mu_1}) \ge
\sqrt{\mu_0\mu_1}(1-\mu_1^{-1}\sqrt{\mu_0\mu_1}) = 
\frac{\sqrt{\mymu}-1}{\alambda }>0.
\end{eqnarray*}
Further, from Step 2. of Algorithm~\ref{alg:computeBA} and Lemma~\ref{lemma:one-minus-xqx} we have
\begin{eqnarray*}
\ell_{m+1}(x) &=& 
\frac{4\mu_0\mu_1}{(\sqrt{\mu_0}+\sqrt{\mu_1})^2}[1-q_m(x)\, x] +
\delta^2 \ell_m(x) \\
&=&
\frac{4\mymu}{\alambda (1+\sqrt{\mymu})^2}[1-q_m(x)\, x] +\delta^2 \ell_m(x) \\
&\ge&
\frac{8\mymu\delta^m}{\alambda (1+\sqrt{\mymu})^2(\delta+\delta^{-1}+2)}+\delta^2 \ell_m(x) .
\end{eqnarray*}
Noticing that $(\delta+\delta^{-1}+2)=\frac{4\mymu}{\mymu-1}$ then leads to:
\begin{equation}\label{eqn:sm1}
\ell_{m+1}(x) \ge 
\frac{2}{\alambda }\delta^{m+1}  +\delta^2 \ell_m(x).
\end{equation}
Clearly, $\ell_{m+1}>0$ if $\ell_m(x) > 0$ 
and a standard induction argument concludes the proof of the lemma. 
\end{proof}
\begin{remark}
From \eqref{eqn:sm1} one can have sharper bounds below on
$\ell_{m+1}(x)$, but we do not pursue these further here. 
\end{remark}
The next lemma is a straightforward corollary of
Lemma~\ref{lemma:one-minus-xqx}. 
\begin{lemma}\label{lemma:positivity}
  Let $q_m(x)$ be the best polynomial approximation of degree at most
  $m$ to $\frac1x$ in $L^{\infty}$-norm on the interval
  $\left[\frac{\alambda }{\mymu},\alambda \right]$, $\mymu > 1$.
  Suppose that $q_m(x)$ is positive on the interval
  $\left[\frac{\alambda }{\mymu},\alambda \right]$.  Then $q_m(x)$ is
  positive on the whole interval $x\in\left(0,\alambda \right]$.
\end{lemma}
\begin{proof}
  We have already shown in the previous lemma that $q_m(x)>q_0(x)>0$,
  for all $m\ge 1$ and
  $x\in\left[0,{\frac{\alambda }{\mymu}}\right]$. Since, by
  assumption $q_m(x)$ is positive on the interval
  $\left[\frac{\alambda }{\mymu},\alambda \right]$ the proof is
  complete.
\end{proof}
In the two-level method convergence estimates in the next section, we
will use the following result (which also includes a sufficient
condition for the positivity of $q_m(x)$).
\begin{lemma}\label{prop:sufficient-for-positivity}
Assume that $\mymu$ and $m$ are such that 
the following inequality holds:
\begin{equation}\label{Pos_Cond_mu}
\left(\frac{\sqrt{\mymu}-1}{ \sqrt{\mymu}+1}\right)^m \leq
\frac{\omega}{\mymu-1}, \quad\mbox{for some}\quad \omega\in (0,2).
\end{equation}
Then the following inequality holds for 
for all $x \in (0,\alambda ]$:
\begin{equation}\label{eqn:bound-on-q}
\frac12\min\left\{\frac{\mymu+1}{\alambda},\frac{2-\omega}{x}\right\}
\leq q_m(x) \leq \frac1x\left(1+ \frac{\omega}{2}\right). 
\end{equation} 
\end{lemma}
\begin{proof}
  \textsf{Lower bound:} We prove first the lower bound when
  $x\in\left[\frac{\alambda }{\mymu},\alambda \right]$.   
Let
  $R_{m+1}(t)$ be the polynomial that has been defined in
  Theorem~\ref{thm:polynomial}.  We use the relation~\eqref{eqn:xqx}
  and Lemma~\ref{lem:Qmp1_bounds}. Note that $-1\le t\le 1$ for   $x\in\left[\frac{\alambda }{\mymu},\alambda \right]$, and we estimate below $xq_m(x)$ as
  follows
\begin{eqnarray*}
xq_m(x) &=& 1 - \frac{2 (-1)^m \delta^m}{(\delta-\delta^{-1})^2} R_{m+1}(t) \ge
1 - \frac{2  \delta^m}{(\delta-\delta^{-1})^2} \vert R_{m+1}(t) \vert \nonumber \\
&\ge& 1 - \frac{2  \delta^m}{(\delta-\delta^{-1})^2} \left( 2 t + \delta + \delta^{-1} \right) \nonumber \\
&\ge& 1 - \frac{2  \delta^m}{(\delta-\delta^{-1})^2} \left( 2 + \delta + \delta^{-1} \right) \nonumber \\
&=& 1 - \frac{2  \delta^m}{\delta + \delta^{-1} -2}= 
1 -\delta^m\frac{\kappa-1}{2} \ge  \frac{2-\omega}{2}.
\end{eqnarray*}
In the last two steps we have used the identity~\eqref{eqn:delta-delta1} and 
the definition of $\delta$, given in~\eqref{eqn:id0}. We thus have
shown that
$q_m(x) \ge \frac{2-\omega}{2x}$ for all
$x\in\left[\frac{\alambda}{\mymu},\alambda \right]$.  Next, 
we apply 
Lemma~\ref{lem:monotonicity} and we have that
\[
q_m(x) \ge q_0(x)  = \frac{\mymu+1}{2\alambda},
\quad \mbox{for}\quad x\in\left[0,\frac{\alambda}{\mymu}\right],
\]
which concludes the proof of the lower bound. 

\textsf{Upper bound:} To prove the upper bound, we need to consider
only the case 
$x\in \left[\frac{\alambda}{\mymu},\alambda\right]$, because  
from Lemma~\ref{lemma:one-minus-xqx} we already know that $xq_m(x) < 1$
for $x\in \left[0,\frac{\alambda}{\mymu}\right]$. 
For $x\in \left[\frac{\alambda}{\mymu},\alambda\right]$, 
we apply an argument analogous to the one for the lower
bound using the relation~\eqref{eqn:xqx} and Lemma~\ref{lem:Qmp1_bounds} (just changing ``$-$'' to ``$+$''): 
\begin{eqnarray*}
xq_m(x) &=& 1 - \frac{2 (-1)^m \delta^m}{(\delta-\delta^{-1})^2}
R_{m+1}(t) 
\leq 
1 +\delta^m\frac{\kappa-1}{2} \le  1+\frac{\omega}{2}.
\end{eqnarray*}
\end{proof}
\begin{remark}
Note that this lemma implies that the polynomial of best approximation
is positive on $[0,\lambda]$ as long as~\eqref{Pos_Cond_mu} is
satisfied with $\omega\in (0,2)$. 
\end{remark}
To conclude this section, we discuss conditions relating $\mymu$ and
the degree of the polynomial $m$ so that \eqref{Pos_Cond_mu} holds.
In what follows, without loss of generality we assume that
$\ln((\mymu-1)/\omega)>1$. In applications (particularly for
analysis of convergence of two-level methods) we are interested in
large values of $\mymu$ (resp. $m$).  Since $\omega\in (0,2)$, such
condition is clearly satisfied for $\mymu>2e+1$.

For fixed and sufficiently large $\mymu$, (as we  assumed above), let $m$ satisfy
\begin{equation}\label{degree}
\frac{\sqrt{\mymu}+1}{2} \ln [(\mymu-1)/\omega]\le m \le 
1+ \frac{\sqrt{\mymu}+1}{2} \ln [(\mymu-1)/\omega].
\end{equation}
We will now show that the lower bound in~\eqref{degree}
implies~\eqref{Pos_Cond_mu} (and therefore also the conclusion of Lemma~\ref{prop:sufficient-for-positivity}). 
Since $0< \delta < 1$ we have
\begin{equation*}
\delta^m =
\left(1-\frac{2}{\sqrt{\mymu}+1}\right)^m \le  
\left[\left(1-\frac{2}{\sqrt{\kappa}+1}\right)^{\sqrt{\kappa}+1}\right]^{\frac12\ln[(\mymu-1)/\omega]}
\end{equation*}
On the other hand,  the function
$(1-2/\xi)^\xi$ is increasing for all $\xi > 2$, and hence
\begin{equation*}
\delta^m <
\left[\lim_{\xi\to\infty}\left(1-\frac{2}{\xi}\right)^\xi
\right]^{\frac12\ln[(\mymu-1)/\omega]}
=
\exp\left(-\ln\frac{\mymu-1}{\omega}\right) = \frac{\omega}{\mymu-1}.
\end{equation*}
Thus, if $\kappa$ is given, the polynomial degree $m$ for
which~\eqref{Pos_Cond_mu} holds is bounded below by the right hand
side of~\eqref{degree}.

In addition, it is easy to show that if~\eqref{degree} holds, then we
also have
\begin{equation}\label{kappa-and-m}
 \frac{1}{\mymu+1}\le c_\omega \left(\frac{\ln
     m}{m}\right)^2,\quad\mbox{with}\quad c_\omega=
\frac12\sup_{\mymu>1;\omega\in(0,2)}\left(\frac{\ln(\mymu/2)}{1+\ln (\mymu/\omega)}\right)^2.
\end{equation}
Note that $c_\omega$ is finite. 
The inequality~\eqref{kappa-and-m}  is seen as follows. Since the logarithm is an increasing function
on its domain we get
\[
\ln m = \ln[\sqrt{\mymu}+1)/2] + \ln\ln[(\mymu-1)/\omega]\ge \frac12\ln(\mymu/2).
\]
Also, from~\eqref{degree}, since $\sqrt{\mymu}+1\ge 2$ we have:
\begin{eqnarray*}
m^2&\le& 
\left(1+ \frac{\sqrt{\mymu}+1}{2} \ln [(\mymu-1)/\omega]\right)^2\le 
\left(\frac{\sqrt{\mymu}+1}{2} (1+\ln [(\mymu-1)/\omega]\right)^2\\
&\le&  \frac{1}{2}(\mymu+1) (1+\ln \mymu/\omega)^2.
\end{eqnarray*}
Combining the last two estimates then gives \eqref{kappa-and-m}. 

\section{An application to two-level
  methods}\label{sect:two-level-methods} 
We consider the linear system of equations 
\begin{equation}
\label{eqn:linear-system} 
A\vek{u}=\vek{f},
\end{equation}
where $A\in \Reals{N\times N}$ is a symmetric and positive definite
matrix, and $\vek{f}\in \Reals{N}$ is a given right hand side vector.
To describe a general two-level multiplicative method,
we denote $\bm{V}=\Reals{N}$, and also introduce a coarse space $V_H$,
$\bm{V}_H\subset V$, $N_H=\operatorname{dim}V_H$, $N_H < N$. In the
following we will always assume that $\bm{V}_H = \operatorname{range}(P)$,
where $P:\Reals{N_H}\mapsto \bm{V}$ and its matrix representation in the
canonical basis of $\Reals{N_H}$ is given by the coefficients in the
expansion of the basis in $\bm{V}_H$ via the basis in $\bm{V}$. Clearly, 
$P$ is a full rank operator and its matrix representation is oftentimes called
\emph{prolongation} or \emph{interpolation} matrix. The restriction of
$A$ on  the coarse space is denoted by $A_H=P^TAP$. 

\subsection{Convergence rate estimates\label{sect:convergence}}
In this subsection we prove convergence estimates for the classical
multiplicative two-level iteration, with polynomial smoother which is
used to define a preconditioner $B \approx A^{-1}$.  In a recent work
\cite{BVV11} the properties of special polynomial smoothers have been
exploited in order to conduct an improved convergence analysis of
smoothed aggregation algebraic multigrid methods. Here, only for
completeness, we include a two-level convergence result presented in
\cite{2011BrezinaM_VassilevskiP-aa}. The only difference is that we
use a polynomial smoother with polynomial defined via
Algorithm~\ref{algorithm: polynomial preconditioning}. 
As in
\cite{2011BrezinaM_VassilevskiP-aa} we show explicit dependence of the
estimates on the degree of the polynomial.

The results up to and including
Theorem~\ref{theorem:abstract-estimate} hold for general SPD $A$,
$\vek{V}$ and $\vek{V}_H$, provided that the smoother is constructed
using the polynomials of best approximation to $1/x$ on a suitably
chosen interval. 

In this subsection, by $\rho(X)$ we denote the spectral radius of a
matrix $X$. If, in addition, $X$ is symmetric and positive definite
matrix, we denote the $X$-norm by $\|\vek{v}\|^2_X = \vek{v}^T X
\vek{v}$.

We define the two-grid (or TG) preconditioner using a classical two-level
algorithm which reads as follows. 
\begin{algorithm}\label{alg:two-level}
Given $\vek{w}\in \vek{V}$ which approximates the solution
of~\eqref{eqn:linear-system} we define the next approximation  $\vek{v}\in \vek{V}$
 to $\vek{u}$ via the following two steps:
\begin{enumerate}
\item Coarse grid correction: $\vek{y} \de \vek{w}  + PA_H^{-1}P^T(\vek{f}-A\vek{w})$
\item Smoothing: $\vek{v} \de \vek{y}+R(\vek{f}-A\vek{y})$.
\end{enumerate}
\end{algorithm}

We assume that $R$ is symmetric and positive definite and $A$-norm
convergent, namely 
\begin{equation}\label{convergent-smoother}
\|I-RA\|_A^2 < 1.
\end{equation}
The error propagation operator
for the two-level iteration above is
\[
E_{\tl} = (I-RA)(I-\pi_A), \quad \pi_A = PA_H^{-1} P^T A.  
\]
We then define the two-level preconditioner as:  
\[
B=(I-E_{\tl}E_{\tl}^{*})A^{-1}.
\]
Here $E_{\tl}^*$ denotes the adjoint with respect to the inner product
defined by $A$.  Introducing  $\bar{R}$ such that  
\begin{equation}\label{Rbar} 
\left( I - \bar{R} A \right) = \left( I - R A \right)^2 
\quad\mbox{and hence}\quad \bar{R} = 2R - R A R .  
\end{equation} 
it is straightforward then to compute that (see, e.g.,
\cite{Panayotsbook}):
\begin{equation}\label{eqn:B}
B = \bar{R} + \left( I - R A \right) P  A^{-1}_H P^T \left( I - A R \right).
\end{equation}
Recall a necessary and sufficient condition for
$R$ to be a convergent smoother in $A$-norm, i.e., \eqref{convergent-smoother} to
hold is that $\bar{R}$ is SPD. 

Our goal will be to prove a convergence rate estimate for the two-level
method with polynomial smoother.  First, let us denote with $D$ the
diagonal of $A$ and set 
\[
R=q_m(D^{-1}A)D^{-1}, 
\] 
where $q_m(x)$ is the polynomial of best approximation to $1/x$,
generated by the Algorithm~\ref{alg:computeBA} on a fixed interval
$[\alambda/\mymu,\alambda]$. Both $\alambda$ and $\mymu$ are to be specified
later. 

One may also write $R$ in the form
\begin{equation}\label{eq:R-and-q}
R = D^{-1/2}q_m(\widehat{A}) D^{-1/2},\qquad\widehat{A}=D^{-1/2}AD^{-1/2}. 
\end{equation}
Using the notation from Section~\ref{sect:properties}, we set
$\alambda=\|\widehat{A}\|_{\ell_\infty}$. 
In what follows, we hold $\alambda$ fixed and we vary $\kappa$
and the degree of the polynomial $m$. However, $\kappa$ and $m$ do not
vary independently and we assume that
$\kappa$ and $m$ satisfy the condition~\eqref{Pos_Cond_mu}. 
With such choice of $\alambda$, $\kappa$ and $m$,  one can easily show that $R$
is a contraction (a convergent smoother) in $A$-norm and we do so by
showing that $\bar{R}$ is SPD, which, as we mentioned earlier,  
is both necessary and sufficient condition
for~\eqref{convergent-smoother} to hold. Clearly,  $\bar{R}$ can be written (see \eqref{Rbar})
as 
\[
\bar{R} = D^{-1/2}[2q_m(\widehat{A}) - q^2_m(\widehat{A})\widehat{A}]D^{-1/2}.
\]
From the upper bound in Lemma~\ref{prop:sufficient-for-positivity} we
immediately get that for all $x\in (0,\lambda]$ we have $xq_m(x)\le
\frac{2+\omega}{2}$. Therefore, for all $\vek{w}\in \vek{V}$ we get
\begin{eqnarray*}
&&\vek{w}^T (2q_m(\widehat{A}) - [q_m(\widehat{A})]^2\widehat{A})\vek{w}
\ge (2-\|xq(x)\|_{\infty,(0,\alambda)})\vek{w}^Tq_m(\widehat{A})\vek{w} \\
&&\ge 
\frac{2-\omega}{2}\vek{w}^Tq_m(\widehat{A})\vek{w}.
\end{eqnarray*}
Applying the inequality above with $\vek{w}=D^{-1/2}\vek{y}$ then
shows that for all $\vek{y}\in \vek{V}$
\begin{equation}\label{eqn:equivalence-1}
\vek{y}^T \bar{R}\vek{y}\ge 
\frac{2-\omega}{2}\vek{y}^T 
R\vek{y}\ge \frac{2-\omega}{2}\min_{x\in (0,\lambda]}q_m(x)\;
\vek{y}^TD^{-1}\vek{y}.  
\end{equation}
From the lower bound in Lemma~\ref{prop:sufficient-for-positivity}, we conclude that $\bar{R}$
is SPD.

We further note that each of
the off-diagonal entries of $(D^{-1/2}AD^{-1/2})$ is less than 1 and
the diagonal entry is equal to 1. Therefore, we have that
\begin{equation}\label{n_z is nnz per row of A}
1\le \|D^{-1/2}AD^{-1/2}\| = \rho(D^{-1/2}AD^{-1/2})\le
\|D^{-1/2}AD^{-1/2}\|_{\ell_\infty} =\alambda \le n_z,
\end{equation}
where $n_z$ is the maximal number of non-zeros in a row of $A$.  

The convergence rate estimates are
derived from the following theorem (two-level version of the XZ-identity, 
cf.~\cite{FVZ2005,Panayotsbook}). 
\begin{theorem}\label{XZ} Assume that $\bar{R}$ is SPD. Then the following
  identity holds:
\begin{equation}\label{eqn:XZ_TL_id}
\vek{v}^T B^{-1} \vek{v} = 
\inf_{\vek{v}_H\in \vek{V}_H} [ \|\vek{v}_H\|_A^2 + \|\vek{v}-\vek{v}_H\|^2_{\bar{R}^{-1}}].
\end{equation}
\end{theorem}
Based on Theorem~\ref{XZ}, we now state and prove a convergence result involving the polynomial smoother. 
\begin{theorem}\label{theorem:abstract-estimate}
Let $A$ be a symmetric positive definite matrix and
  $D$ be its diagonal. Let
  $\alambda=\|D^{-1/2}AD^{-1/2}\|_{\ell_\infty}$, and also  
  $\kappa>1$ and $m$ satisfy~\eqref{Pos_Cond_mu}. If $R=q_m(D^{-1}A)D^{-1}$, with
  $q_m(x)$ the polynomial of best approximation to $1/x$ on the
  interval $[\alambda/\mymu,\alambda]$, then the following
  estimate holds for all $\vek{v}\in \vek{V}$:
\begin{equation}\label{eqn:abstract-estimate}
\vek{v}^T B^{-1} \vek{v} \le 
\frac{4}{(2-\omega)}
\inf_{\vek{v_H}\in \vek{V_H}}
\left[\|\vek{v}_H\|_A^2 +
\frac{\alambda}{(\mymu+1)}
\|\vek{v}-\vek{v}_H\|_{D}^2 + 
\frac{1}{2-\omega}\|\vek{v}-\vek{v}_H\|^2_{A}\right]\;.
\end{equation}
\end{theorem}
\begin{proof}
First, we see that from \eqref{eqn:equivalence-1} we have that 
\begin{equation}\label{eqn:equivalence-2}
\vek{y}^T \bar{R}\vek{y}\ge 
\frac{2-\omega}{2}\vek{y}^T R\vek{y}\quad\mbox{and hence}\quad
\vek{y}^T \bar{R}^{-1}\vek{y}\le \frac{2}{2-\omega}\vek{y}^T R^{-1}\vek{y}.
\end{equation}
Under the assumptions we made in the statement of the theorem
we can apply Lemma~\ref{prop:sufficient-for-positivity},  and get that for all $x\in (0,\lambda]$, 
\begin{equation}\label{eqn:bound-on-q-1}
\frac{1}{q_m(x)}
\leq 2\max
\left\{\frac{\alambda}{\mymu+1},\frac{x}{2-\omega}\right\}
\leq 
\left(\frac{2\alambda}{\mymu+1}+\frac{2x}{2-\omega}\right). 
\end{equation} 
Since $\widehat{A}$ and $q_m(\widehat{A})$ commute, and have the same
set of orthonormal eigenvectors, we have that  for all $\vek{w}\in \vek{V}$ we have
\begin{equation*}
\vek{w}^T [q_m(\widehat{A})]^{-1}\vek{w}
\leq 
\frac{2\alambda}{\mymu+1}\|\vek{w}\|_{\ell_2}^2 + \frac{2}{2-\omega}\|\vek{w}\|^2_{\widehat{A}}.
\end{equation*}
Taking $\vek{y}=D^{1/2}\vek{w}$ in the inequality above and using the
estimate given in~\eqref{eqn:equivalence-2}
\begin{equation}\label{polynomial smoothing property}
\vek{y}^T \bar{R}^{-1}\vek{y}\le \frac{2}{2-\omega}
\vek{y}^T R^{-1}\vek{y}
\leq 
\frac{4}{(2-\omega)}\left[\frac{\alambda}{(\mymu+1)}
\|\vek{y}\|_{D}^2 + \frac{1}{2-\omega}
\|\vek{y}\|^2_{A}\right].
\end{equation}
The proof is concluded by taking $\vek{y}=(\vek{v}-\vek{v}_H)$ and
applying~Theorem~\ref{XZ}.
\end{proof}
Without loss of generality, we set now $\omega=1$  and use that 
in equation~\eqref{kappa-and-m}
$c_\omega=c_1\le \frac12$.  
The estimate in
the Theorem~\ref{theorem:abstract-estimate} takes the form. 
\begin{corollary}\label{coro:abstract-estimate} 
Under the assumptions of
  Theorem~\ref{theorem:abstract-estimate}, with $\omega=1$ we have
\begin{equation}\label{eqn:abstract-estimate-1}
\vek{v}^T B^{-1} \vek{v} \le 
4\inf_{\vek{v_H}\in \vek{V_H}}
\left[\|\vek{v}_H\|_A^2 +
\frac{\alambda}{(\mymu+1)}
\|\vek{v}-\vek{v}_H\|_{D}^2 + \|\vek{v}-\vek{v}_H\|^2_{A}\right]\;.
\end{equation}
In addition, if $\mymu$ and $m$ satisfy~\eqref{degree}
we have 
\begin{equation}\label{eqn:abstract-estimate-1}
\vek{v}^T B^{-1} \vek{v} \le 
2\inf_{\vek{v_H}\in \vek{V_H}}
\left[\|\vek{v}_H\|_A^2 +
\frac{\alambda\ln^2m}{m^2}
\|\vek{v}-\vek{v}_H\|_{D}^2 + 2\|\vek{v}-\vek{v}_H\|^2_{A}\right].
\end{equation}
\end{corollary}

To stress the fact that estimate \eqref{polynomial smoothing property} is purely algebraic, we formulate it
separately, as this is our main new result.

\begin{theorem}\label{theorem: polynomial smoothing property}
Let $A$ be an s.p.d. matrix and $D$ a given s.p.d. preconditioner for $A$ such that
$\|D^{-\frac{1}{2}} A D^{-\frac{1}{2}}\| \le \lambda$. Consider the polynomial preconditioner
\begin{equation*}
R = q_m(D^{-1}A) D^{-1},
\end{equation*}
where $q_m$ is the polynomial of best approximation of $1/x$ over the interval $\left [\frac{\lambda}{\kappa},\; 
\lambda \right ]$. The parameter $\kappa$ is chosen depending on $m$ such that \eqref{Pos_Cond_mu} holds for a given
$\omega \in (0,2)$.
Then the following smoothing property holds for $R$ and its symmetrized version ${\overline R}$ (see \eqref{Rbar}):
\begin{equation*}
\frac{2-\omega}{2}\;\bv^T {\overline R}^{-1} \bv \le \bv^T R^{-1} \bv
\le \frac{2\lambda}{\kappa+1}\;\bv^TD\bv+ \frac{2}{2-\omega}\;\bv^T A\bv.
\end{equation*}
In addition, if $\kappa$ and $m$ satisfy \eqref{degree}, we have
\begin{equation*}
\frac{2-\omega}{2}\;\bv^T {\overline R}^{-1} \bv \le \bv^T R^{-1} \bv
\le 2\lambda\;\frac{\ln^2 m}{m^2}\; \bv^TD\bv+ \frac{2}{2-\omega}\;\bv^T A\bv.
\end{equation*}
\end{theorem}

\subsection{Two-level method for discretized PDE\label{two-level-pde}}

In this section we apply the abstract two--level result to the case of
a two-level iterative method with large coarsening ratio for the
solution of a system of linear algebraic equations arising from a
discretization of scalar elliptic equation with heterogeneous
coefficients similarly to the presentation in
\cite{2011BrezinaM_VassilevskiP-aa}, now for the case of a different
polynomial smoother from Theorem~\ref{theorem: polynomial smoothing
  property}.  We consider the following variational problem: Find $u\in
H^1_D(\Omega)$, for a given polygonal (polyhedral) domain $\Omega
\subset \Reals{d}\ (d=2\text{ or } 3)$ and a source term $f \in
L_2(\Omega)$, such that
\begin{equation}\label{eqn:model-problem-with-jumps}
\begin{array}{rcl} a(u,\;v) \equiv 
\int_{\Omega} \alpha(\bm{x}) \;\nabla u \cdot \nabla v  &=& 
\int_\Omega f(\bm{x})v(\bm{x})  = (f,\;v)\;,  
\quad\mbox{for all}\quad v\in H_D^1(\Omega).
\end{array}
\end{equation} 
Here, $\Omega\subset \Reals{d}$ $d=2,3$ is a given domain whose
boundary $\Gamma=\partial\Omega$ is partitioned as
$\Gamma=\Gamma_D\cup\Gamma_N$. We assume that $\Gamma_D\neq\emptyset$
is closed as a subset of $\Gamma$ and also has a nonzero $(d-1)$
dimensional measure. We refer to $\Gamma_D$ as the Dirichlet part of the
boundary and $\Gamma_N$ as the Neumann part of the boundary.  In the
variational problem~(\ref{eqn:model-problem-with-jumps}), $H^1_D(\Omega)$
denotes the space of functions in $H^1(\Omega)$ whose traces
vanish on $\Gamma_D$.  

We are interested in the case when the diffusion
coefficient $\alpha = \alpha(\bm{x})$ is a piecewise constant
function, that may have large variations within $\Omega$. We thus
assume that $\bar\Omega=\cup_{l=1}^{m_0} \bar{\mathcal{Y}}_l$, with
polygonal (polyhedral) subdomains $\mathcal{Y}_l$, and that
$\alpha(\bm{x})=\alpha_l$, for all $\bm{x} \in \mathcal{Y}_l$ and
$l=1,\ldots, m_0$.  We introduce the following energy norm
\begin{equation}\label{eqn:weighted-l2-norm-definition}
\|v\|_{a}^2 = \int_{\Omega} \alpha(\bm{x})|\nabla v|^2 =
\sum_{l=1}^{m_0}\alpha_l\int_{\mathcal{Y}_l}|\nabla v|^2.
\end{equation}
We  also need the weighted $L_2$ norm
\begin{equation}\label{eqn:l2-norm-definition}
\|v\|_{0,\alpha}^2 = \int_{\Omega}\alpha(\bm{x}) v^2 =
\sum_{l=1}^{m_0}\alpha_l\int_{\mathcal{Y}_l}v^2.
\end{equation}
We consider a standard discretization of the variational
problem~\eqref{eqn:model-problem-with-jumps} with piecewise linear
continuous finite elements.  To define the finite element spaces and
the approximate solution, we assume that we have a locally quasi--uniform, 
simplicial triangulation $\mathcal{T}_h$ of $\Omega$. We assume that this 
triangulation also resolves $\mathcal{Y}_l$, namely, for $l=1,\ldots,m_0$ we have:
\begin{equation}\label{eqn:alignment-property} 
\bar{\Omega} = \cup_{\tau\in \mathcal{T}_h} \tau, \quad \bar{\mathcal{Y}}_l =
\cup_{\tau\in \mathcal{T}_{Y,l}} \tau,
\end{equation}
where $\mathcal{T}_{Y,l} \subset \mathcal{T}_h$, for $l=1,\ldots,m_0$.
The standard space of piecewise linear (w.r.t $\mathcal{T}_h$) and
continuous functions  vanishing on the boundary of $\Omega$ is denoted by $V_h$.  

The  discrete problem then reads: Find $u\in V_{h}$ such that 
\begin{equation}\label{eqn:discrete}
a(u,v)=(f,v), \quad \mbox{for all} \quad v\in V_{h}.
\end{equation}
The notation and  constructions in the previous section are suitable
for the finite element setting as well. Indeed, a coarse space 
corresponding to $\vek{V}_H$ (denoted here with $V_H$) as
$V_H = \operatorname{range}(P)$, with the same $P$ as before, but this
time representing the coefficients in the expansion of the basis in
$V_H$, $\{\varphi_j^H\}_{j=1}^{N_H}$ via the canonical Lagrange
basis $\{\varphi_j\}_{j=1}^{N}$ in $V_h$.  
Evaluating the bilinear form on the basis for $V_h$ and the basis for
$V_H$ defines the stiffness matrix $A$ and the matrix $A_H$:
\begin{eqnarray*}
A_{kj}=a(\varphi_j,\varphi_k), \quad (A_H)_{kj} = (P^TAP)_{jk} =a(\varphi^H_j,\varphi^H_k). 
\end{eqnarray*}
According to the considerations in the previous section, we use bold
face to represent vectors of degrees of freedom and normal font for
functions. Thus a function $v\in V_h$ is represented by the vector
$\vek{v}\in \vek{V}$.

We make the following assumption for the stability and approximation
properties of the coarse function space $V_H$.
\begin{itemize}
\item \textsf{Approximation and stability assumption:} 
For any $v\in V_h$ there exists $v_H\in V_H$ such that
\begin{equation}\label{eqn:approx}
H^{-2}\|v-v_H\|^2_{0,\alpha} +\|v-v_H\|^2_{a} \le c_{as} \|v\|^2_a, 
\end{equation}
where $H$ is the diameter of the support of a typical basis function
in $V_H$, and the constant $c_{as}$ is independent of the variations
of the coefficient $\alpha(\vek{x})$.
\end{itemize}

Construction of coarse spaces satisfying this assumption is possible as already mentioned, 
and we refer
to \cite{2010GalvisJ_EfendievY-aa},
\cite{2011ScheichlR_VassilevskiP_ZikatanovL-aa}, and earlier \cite{1999BrezinaM_HebertonC_MandelJ_VanekP-aa}
 as modified recently in \cite{2011BrezinaM_VassilevskiP-aa}
for such constructions. 

We next introduce a well-known inequality relating the weighted $L^2$
norm on the function space $V_h$ and the norm provided by the diagonal of
the stiffness matrix on the space of degrees of freedom (nodal values of
the piece-wise linear functions).  Let $\{\lambda_{j,T}\}_{j=1}^{d+1}$
  be the barycentric coordinates in an element $T\in \mathcal{T}_h$ and
  $\alpha_T$ be the value of the coefficient on $T$ (recall that
  $\alpha(\bm{x})$ is piece-wise constant).  Let $v\in V_h$
  with corresponding vector of degrees of freedom $\vek{v}\in
  \vek{V}$. We have the following simple inequality
\begin{eqnarray*}
\|\vek{v}\|_D^2 &=& \sum_{T\in \mathcal{T}_h}
\alpha_{T}\sum_{j=1}^{d+1}\vek{v}^2_{j,T}|\nabla\lambda_{j,T}|^2
\le \sum_{T\in \mathcal{T}_h} c_Th_T^{-2}\alpha_T\sum_{j=1}^{d+1}\vek{v}^2_{j,T}|\lambda_{j,T}|^2\\
&\le& \sum_{T\in \mathcal{T}_h} h_T^{-2}\alpha_T
c_Tc_{M,T}\|v\|_{L^2(T)}^2.
\end{eqnarray*}
In the inequalities above, we have used standard inverse inequality,
and also that the local mass matrix for an element $T$ is equivalent
to its diagonal with a bound $c_{M,T}$ independent of the coefficient
variation. Finally, 
\begin{equation}\label{eqn:mass}
\|\vek{v}\|_D^2 \le c_Mh^{-2}\|v\|^2_{0,\alpha}, \quad\mbox{with}\quad c_M=\max_{T\in\mathcal{T}_h} c_Tc_{M,T}. 
\end{equation}
It is also clear that $\|\vek{v}\|_A = \|v\|_a$ by the definition of
the stiffness matrix.  

We now formulate the spectral equivalence result for the two-level
method when applied to the discretized
PDE~\eqref{eqn:model-problem-with-jumps}.
\begin{theorem}\label{theorem: contrast independent TG convergence}
  Let $\mymu>1$, and $m$ be such that~\eqref{Pos_Cond_mu} holds with
  $\omega=1$. Assume that $V_H$ is such that the approximation and
  stability assumption holds. Then we have the following spectral
  equivalence (for $n_z$ see \eqref{n_z is nnz per row of A}):
\begin{equation}
\vek{v}^TA\vek{v} \le \vek{v}^TB^{-1}\vek{v} \le K_{TG}\;
\vek{v}^TA\vek{v},\quad K_{TG}=1+ 4c_{as}\left[\frac{c_Mn_z}{(\mymu+1)} \left(\frac{H}{h}\right)^2 + 1\right]. 
\end{equation}
Moreover, if $\kappa$ (or equivalently the degree of the polynomial
$m$) is sufficiently large the spectral equivalence is uniform with
respect to mesh parameters and coefficient variation. 
\end{theorem}
\begin{proof} The proof is the same as the one given in \cite{2011BrezinaM_VassilevskiP-aa} however with a different 
smoothing property provided by Theorem \ref{theorem: polynomial smoothing property}. 

The lower bound is immediate, since $E_{\tl}$ is a contraction in
$A$-norm. The upper bound follows directly from
Corollary~\ref{coro:abstract-estimate} together used in conjunction
with the simple inequalities relating the function space $V_h$ and
$\vek{V}$ (see~\eqref{eqn:mass}). Given $v\in V_h$, let $\vek{v}\in
\vek{V}$ be the corresponding vector of degrees of freedom. We have
\begin{eqnarray*}
\vek{v}^T B^{-1} \vek{v} &\le&  
4\inf_{\vek{v_H}\in \vek{V_H}}
\left[\|\vek{v}_H\|_A^2 +
\frac{\alambda}{(\mymu+1)}
\|\vek{v}-\vek{v}_H\|_{D}^2 + \|\vek{v}-\vek{v}_H\|^2_{A}\right]\\
& \le &
4\inf_{\vek{v_H}\in \vek{V_H}}\left[\|\vek{v}_H\|_A^2 +
\frac{c_Mn_zh^{-2}}{(\mymu+1)} \|v-v_H\|_{0,\alpha}^2 + \|v-v_H\|^2_{a}\right]\\
& \le &
\left[ 1+ \frac{4c_{as}c_Mn_z}{(\mymu+1)} \left(\frac{H}{h}\right)^2 + 4c_{as}\right]\|v\|^2_{a}
= K_{TG}\; \vek{v}^T A \vek{v}.
\end{eqnarray*}
Clearly, for $(\sqrt{\mymu}+1)\ge \frac{H}{h}$, and 
$m$ satisfying~\eqref{degree}, for example, $m\ge \frac{H}{h} \ln (H/h)$,
the spectral
equivalence is uniform with respect to mesh size and coefficient
variation.  
\end{proof}


\section{Choice of coarse spaces and numerical tests}
\label{section:numerical experiments}
In this section, we present a number of tests that illustrate the
robustness of the two--level methods with the polynomial smoother
analyzed in the present paper all in accordance with Theorem
\ref{theorem: contrast independent TG convergence}.  We consider the
second order elliptic equation~\eqref{eqn:model-problem-with-jumps}
with a mixture of Neumann and Dirichlet boundary conditions. The
Dirichlet boundary conditions are imposed on the ``east'' and ``west''
vertical boundaries, i.e.  $\Gamma_D=\GEa \cup \GWe$ of $\Omega$. As
we pointed out, the coefficient $\alpha(\bx)$ is piecewise constant
and we assume that the fine triangulation of $\Omega$ is aligned with
(resolves) all the coefficient discontinuities. In
Fig.~\ref{fig:coeff_distrib} we show an example of a fine grid
$\mathcal{T}_h$, aligned with discontinuities.
\begin{figure}[!htb]
\centering
\includegraphics[width=0.45\textwidth]{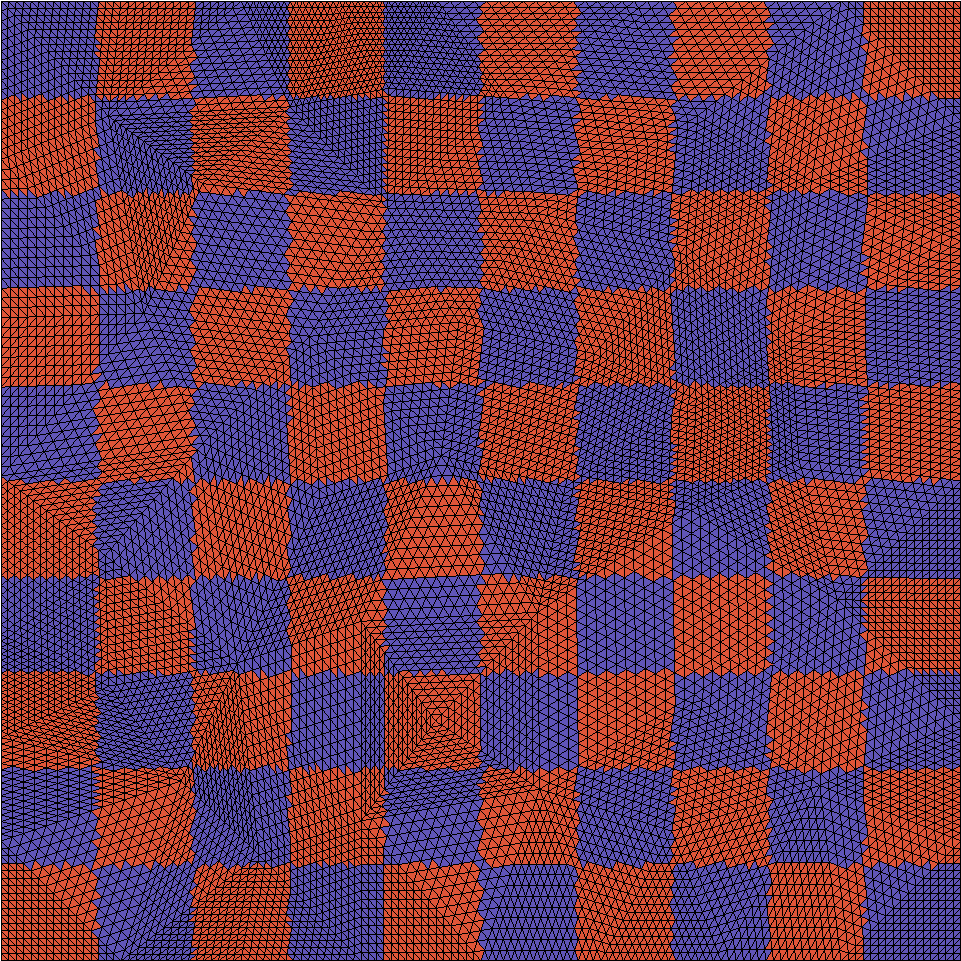}
\caption{Checkerboard coefficient distribution on a mesh with $25600$ elements and $13041$ vertices.}
\label{fig:coeff_distrib}
\end{figure}
\subsection{Coarse spaces}
We use element agglomeration to define ``coarse elements'' as
illustrated in Fig.~\ref{fig:AEs}.  and a variant of the spectral AMGe
method (see, e.g. \cite{Panayotsbook}) in the form presented in
\cite{2011BrezinaM_VassilevskiP-aa}.
\begin{figure}[!htb]
  \centering \subfloat[]
{\includegraphics[width=0.4\textwidth]{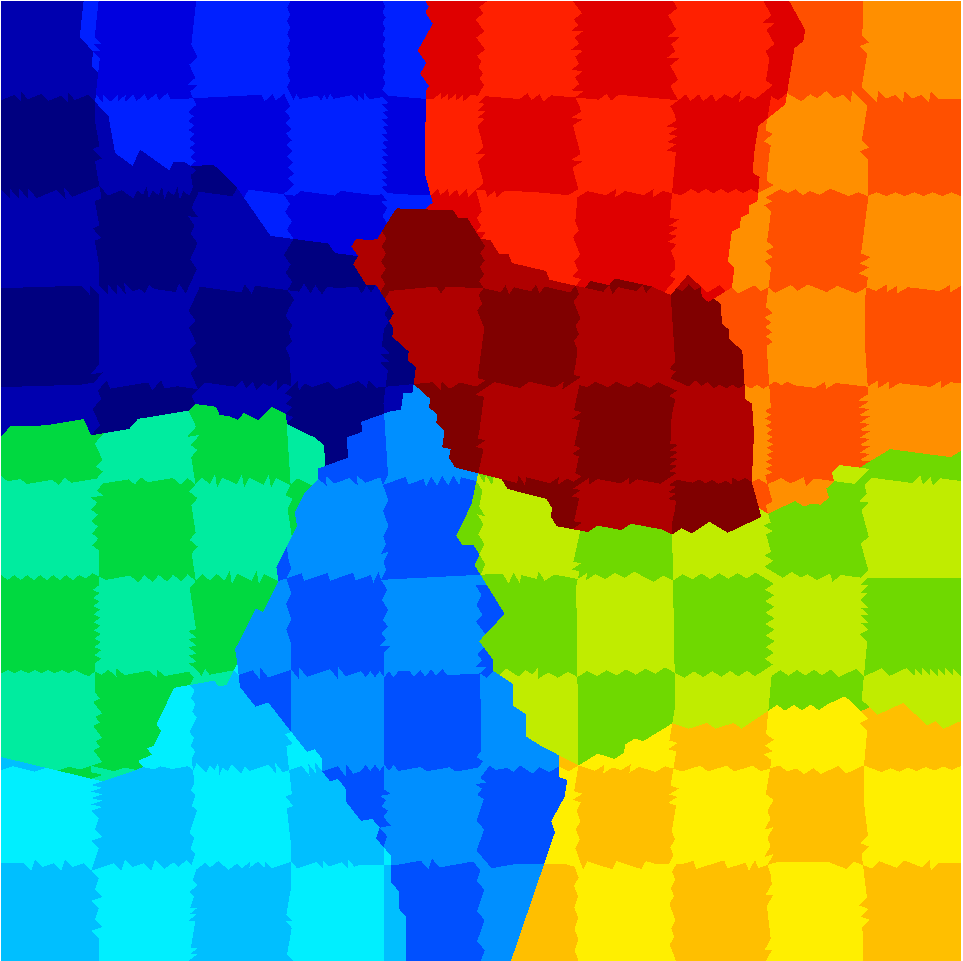}}\hspace*{0.1\textwidth}
  \subfloat[]
{\includegraphics[width=0.4\textwidth]{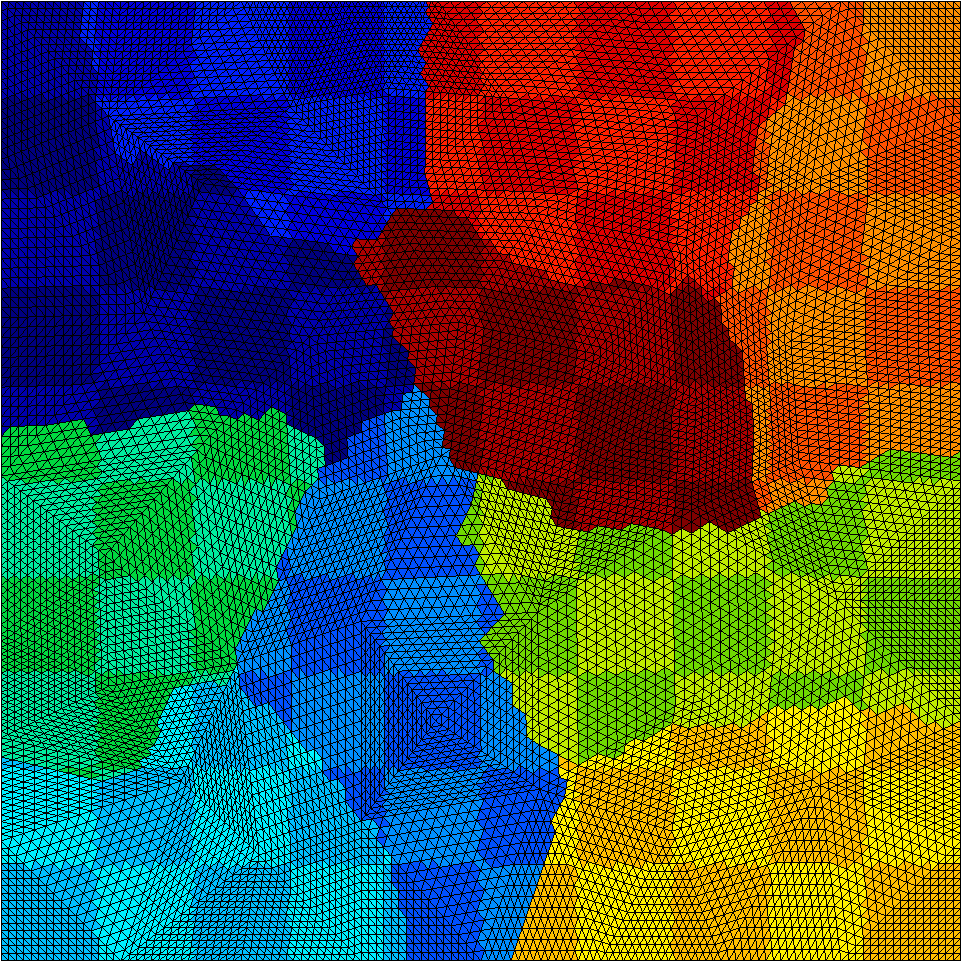}}
\caption{(a) An example of $10$ aggregates; (b) the corresponding
  element agglomerates (unions of fine grid elements) constructed on a
  mesh with $6400$ elements and $3321$ vertices. 
The distribution of discontinuity of the PDE coefficient is not resolved by the 
agglomerates.}
\label{fig:AEs}
\end{figure}
Briefly the main steps in such  coarse space construction are:
\begin{itemize}
\item Partitioning of the degrees of freedom as a union of
  non-overlapping sets, $\{\A\}$ called \emph{aggregates}.  This is achieved by first partitioning the set of elements into agglomerated elements $\{\tau\}$ (union of fine-grid elements). 
We use graph partitioner (metis) applied to the graph having 
vertices the fine-grid elements with edges between two elements if they share a common interface. 
Then, we form aggregates $\A$, where each aggregate (a set of fine degrees of freedom) 
corresponds a unique agglomerated element
  $\tau = \tau_{\A}$ by distributing the shared fine degrees of freedom (fine-grid 
element vertices belonging to two or more 
agglomerated elements) to a unique aggregate. 

\item Constructing a tentative interpolation matrix ${\overline P}$,
  defined for an agglomerate $\tau$.  Consider the local
  generalized eigenproblem,
$$A_{\tau} \bm{\varphi}_k  = \theta_k D_{\tau} \bm{\varphi}_k,$$ 
where $A_{\tau}$ is the local stiffness matrix corresponding to
the agglomerated element $\tau$, $D_{\tau}$ is its diagonal and 
$\theta_1\le \theta_2\le\ldots \le\theta_{n_{\A}}$ with $n_{\A}=|\A|$ (cardinality of
$\A$). Given a \emph{spectral tolerance}
$\theta$, we select the eigenvectors
$\{\varphi_k\}_{k=1}^{n_{\theta}}$, where $n_{\theta}$ is the largest
integer for which the inequality $\theta_{n_\theta} < \theta$ holds.
Extended by zero outside each $\A$, the vectors 
$\{\bm{\varphi_k}\}_{k=1}^{n_{\theta}}$ form $n_\theta$ columns of
the global tentative interpolation operator ${\overline P}$.
\item 
Constructing the coarse space as the range of the interpolation matrix
$P$, which is defined as 
\begin{equation*}
P = s_m\left (\lambda^{-1} D^{-1}A\right ) {\overline P}.
\end{equation*}
Here, as in the previous section,  
$D$ is the diagonal of $A$,  $\lambda \ge \|D^{-\frac{1}{2}} A
D^{-\frac{1}{2}}\|$ (e.g. $\lambda=\|D^{-1/2}AD^{-1/2}\|_{\infty}$), 
and  $s_{m}(t)$ is the
smoothed aggregation (SA) polynomial (cf., e.g., \cite{BVV11})
\begin{equation*}
s_m(t) = \frac{(-1)^m}{(2m+1)}\;\frac {T_{2m+1}(\sqrt{t})}{\sqrt{t}}.
\end{equation*}
\end{itemize}

\subsection{Numerical tests}
We recall some of the notations and definitions which are used in the
tables and figures in this section. 
\begin{itemize}
\item $N$ is the number of fine grid degrees of freedom; 
\item $N_H$ is the number of coarse degrees of freedom; 
\item $nnz(X)$ is the number of the nonzero elements in a matrix $X$;
\item ${\widetilde {\varrho}}_{TG}$ is the asymptotic convergence
  factor of the two grid method;
\item $oc(B)$ is the \emph{operator complexity} measure of the two-grid preconditioner $B$,
  defined as $oc(B)=\frac{nnz(A)+nnz(A_H)}{nnz(A)}$.
\end{itemize}

The first set of experiments are on a mesh with $102,400$ elements and
$N=51,681$ vertices using $300$ agglomerated elements (AEs).  We stop
the iterations when the relative preconditioned residual norm is
reduced by a factor of $\varepsilon = 10^{-8}$.  The piecewise
constant coefficient $\alpha(\bx)$ is distributed in a checkerboard
fashion with values $1$ and $10^6$ as illustrated in
Fig. \ref{fig:coeff_distrib}.

The experiments are performed for $m = 2, 4, 6, 8$, $a\equiv
\frac{1}{\kappa}= 0.158, 0.085, 0.055, 0.04$, and $a = 0.2, 0.1, 0.08,
0.06$ respectively. They are chosen such that the
inequality~\eqref{Pos_Cond_mu} (with $\omega=1$) holds:
\begin{equation*}
\left (\frac{\sqrt{\mymu}-1}{\sqrt{\mymu}+1} \right )^m =
\left (\frac{1 - \sqrt{a}}{1+\sqrt{a}} \right )^m < \frac{a}{1-a} = \frac{1}{\kappa-1}, 
\quad a = \frac{1}{\kappa}.
\end{equation*}
The same degree $m$ is used for the polynomial smoother in the
two-level algorithm and the smoother of the tentative interpolation
matrix (it is smoothed out by $s_m \lp \inv{D}A \rp$).  The number of
non-zero entries of $A$, is $nnz(A)=359,841$.  We also show how the
spectral tolerance $\theta$ and the polynomial degree $m$ influences
the convergence versus \emph{operator complexity}. The results are
presented in~Tables~\ref{table:m=2}--\ref{table:m=8}.  It is evident
from the results that the method can become fairly fast (in terms of
convergence factors) at the expense of large operator complexity.

\begin{table}[!htb]
\centering
\begin{tabular}{|l|r|r|r|r|r|}
\cline{5-6}
\multicolumn{1}{c}{} & \multicolumn{1}{c}{} & 
\multicolumn{1}{c}{} & \multicolumn{1}{c}{} &
\multicolumn{2}{|c|}{$\trhotg$}\\ \hline
\multicolumn{1}{|c}{$\theta$}   & \multicolumn{1}{|c}{$N_H$} & 
\multicolumn{1}{|c}{$nnz(A_H)$} & \multicolumn{1}{|c}{$oc(B)$} &
\multicolumn{1}{|c}{{\small $a=0.158$}} & \multicolumn{1}{|c|}{{\small $a=0.2$}}\\
\hline
  0.010   & 774 & 15,930        & 1.04 & 0.995 & 0.995 \\
\hline
0.077 & 3,629 & 342,515 & 1.95 & 0.879 & 0.889 \\
\hline
0.149 & 6,557 & 1,115,207 & 4.10 & 0.393 & 0.492 \\
\hline
\end{tabular}
\caption{Two-grid convergence, $m = 2$. \label{table:m=2}}
\end{table}

\begin{table}[!htb]
\centering
\begin{tabular}{|l|r|r|r|r|r|}
\cline{5-6}
\multicolumn{1}{c}{} & \multicolumn{1}{c}{} & 
\multicolumn{1}{c}{} & \multicolumn{1}{c}{} &
\multicolumn{2}{|c|}{$\trhotg$}\\ \hline
\multicolumn{1}{|c}{$\theta$}   & \multicolumn{1}{|c}{$N_H$} & 
\multicolumn{1}{|c}{$nnz(A_H)$} & \multicolumn{1}{|c}{$oc(B)$} &
\multicolumn{1}{|c}{{\small $a=0.085$}} & \multicolumn{1}{|c|}{{\small $a=0.1$}}\\
\hline
0.010 & 774 & 22,092 & 1.06 & 0.985 & 0.986 \\
\hline
0.077 & 3,629 & 472,907 & 2.31 & 0.531 & 0.538 \\
\hline
0.149 & 6,557 & 1,538,845 & 5.28 & 0.188 & 0.084 \\
\hline
\end{tabular}
\caption{Two-grid convergence when $m = 4$.\label{table:m=4}}
\end{table}

\begin{table}[!htb]
\centering
\begin{tabular}{|l|r|r|r|r|r|}
\cline{5-6}
\multicolumn{1}{c}{} & \multicolumn{1}{c}{} & 
\multicolumn{1}{c}{} & \multicolumn{1}{c}{} &
\multicolumn{2}{|c|}{$\trhotg$}\\ \hline
\multicolumn{1}{|c}{$\theta$}   & \multicolumn{1}{|c}{$N_H$} & 
\multicolumn{1}{|c}{$nnz(A_H)$} & \multicolumn{1}{|c}{$oc(B)$} &
\multicolumn{1}{|c}{{\small $a=0.055$}} & \multicolumn{1}{|c|}{{\small
    $a=0.08$}}\\ 
\hline
0.010 & 774 & 29,448 & 1.08 & 0.965 & 0.969 \\
\hline
0.077 & 3,629 & 636,671 & 2.78 & 0.205 & 0.179 \\
\hline
0.149 & 6,557 & 2,074,291 & 6.76 & 0.202 & 0.026\\
\hline
\end{tabular}
\caption{Two-grid convergence when $m = 6$. \label{table:m=6}}
\end{table}

\begin{table}[!htb]
\centering
\begin{tabular}{|l|r|r|r|r|r|}
\cline{5-6}
\multicolumn{1}{c}{} & \multicolumn{1}{c}{} & 
\multicolumn{1}{c}{} & \multicolumn{1}{c}{} &
\multicolumn{2}{|c|}{$\trhotg$}\\ \hline
\multicolumn{1}{|c}{$\theta$}   & \multicolumn{1}{|c}{$N_H$} & 
\multicolumn{1}{|c}{$nnz(A_H)$} & \multicolumn{1}{|c}{$oc(B)$} &
\multicolumn{1}{|c}{{\small $a=0.04$}} & \multicolumn{1}{|c|}{{\small
    $a=0.06$}}\\ 
\hline
0.010 & 774 & 37,618 & 1.10 & 0.926 & 0.933 \\
\hline
0.077 & 3,629 & 808,357 & 3.25 & 0.197 & 0.111 \\
\hline
0.149 & 6,557 & 2,632,755 & 8.32 & 0.193 & 0.028 \\
\hline
\end{tabular}
\caption{Two-grid convergence when $m = 8$. \label{table:m=8}}
\end{table}

In the last experiment shown in Table~\ref{table:table9}, we illustrate the behavior of the
method with respect to varying the contrast $10^c$ again distributed
in a checkerboard fashion. As it is clearly seen, the two-grid method
exhibits very good uniform two-grid convergence with operator complexity
less than two.


\begin{table}[!htb]
\centering
\begin{tabular}{@{\extracolsep{-3pt}}|c|c|c|c|c|c|c|c|c|c|}
\hline
$c$ & -12 & -9 & -6 & -3 & 0 & 3 & 6 & 9 & 12 \\
\hline
$N_H$ & 2336 & 2336 & 2336 & 2339 & 2322 & 2322 & 2322 & 2322 & 2322\\
$oc(B)$ & 
1.94 & 1.94 & 1.94 & 1.94 & 1.93 & 1.93 & 1.93 & 1.93 & 1.93\\
\hline
$n_{it}$ & 17 & 17 & 17 & 17 & 17 & 16 & 16 & 16 & 16 \\
\hline
$\trhotg$ & 0.219 & 0.219 & 0.219 & 0.219 & 0.219 & 0.200 & 0.198 & 0.197 & 0.198 \\ 
\hline
\end{tabular}
\caption{Contrast independent two-grid convergence; coefficient jumps
are $10^c$. The method corresponds to spectral threshold  $\theta =
0.045$, $m = 8$, and $a = 0.04$.\label{table:table9}}
\end{table}

\section*{Acknowledgments}
The authors thank Delyan Kalchev from Sofia University ``St. Kliment
Ohridski'' for his help with the numerical experiments.

\appendix

\section{}
The proof of Theorem~\ref{thm:polynomial} presented in this section is
based on an equivalent result given
in~\cite[p.~33,~Equation~(4.25)]{1967MeinardusG-aa}.  Let us also remark that in
this section our considerations are on the interval $[-1,1]$ and in
addition, by \emph{best polynomial approximation} we mean the best
polynomial approximation in the norm $\|\cdot\|_{\infty}$ on $[-1,1]$.

\subsection{An approximation result equivalent to Theorem~\ref {thm:polynomial}\label{sect:gm-thm}} 

We now formulate the result in~\cite{1967MeinardusG-aa} in the notation
introduced earlier and show how the result in
Theorem~\ref{thm:polynomial} can be derived from
\cite[p.~33,~Equation~(4.25)]{1967MeinardusG-aa}.

\begin{theorem}[G. Meinardus,~\cite{1967MeinardusG-aa}]\label{thm:gm-polynomial}
  The polynomial $\widetilde{Q}_m\in \mathcal{P}_{m}$, of degree less
  than or equal to $m$, which furnishes the best approximation to
  $\frac{1}{t-a}$, $a> 1$ on $[-1,1]$ is given by:
\begin{equation*}
\widetilde{Q}_m(t) = \frac{1}{t-a}\left(1 - 
\frac{(a - \sqrt{a^2-1})^{m}}{a^2-1}
\widetilde{R}_{m+1}(t)\right),
\end{equation*}
where
\begin{equation*}
\widetilde{R}_{m+1}(t) = 
\left[(a t-1)T_{m}(t)+\frac{\sqrt{a^2-1}}{m}(t^2-1)
T_{m}^\prime(t) \right].
\end{equation*}
\end{theorem}

The result we have just stated is for the best approximation to the
function $\frac{1}{t-a}$, while to prove Theorem~\ref{thm:polynomial}
we need such result for $\frac{1}{t+a}$.  It is however easy to show
that Theorem~\ref{thm:gm-polynomial} also provides the best
polynomial approximation to $\frac{1}{t+a}$.  Indeed,
note that for any polynomial $p(t)$ of degree less than or equal to
$m$, and for all $t\in [-1,1]$, there holds
\begin{equation*}
p(-t) - \frac{1}{(-t)-a} = -\left(-p(-t) - \frac{1}{t+a}\right).
\end{equation*}
Further, for a function $g(t)$ continuous on $[-1,1]$ we also have,
\begin{equation*}
\max \{g(t)~|~t\in[-1,1]\} = 
\max \{g(-t)~|~t\in[-1,1]\},
\end{equation*}
These two identities give that
\begin{equation*}
\left\|p - \frac{1}{t-a}\right\|_{\infty} =
\max_{t\in[-1,1]}
\left|p(-t) - \frac{1}{(-t)-a}\right| = 
\max_{t\in[-1,1]}\left|(-p(-t)) - \frac{1}{t+a}\right|.
\end{equation*}
Since $p$ was an arbitrary polynomial of degree less than or equal to
$m$, we may take the infimum over all $p\in \mathcal{P}_m$. According
to Theorem~\ref{thm:gm-polynomial} the left side is minimized for
$p(t)=\widetilde{Q}_m(t)$. Therefore the right side should also be
minimized for $p(t)=\widetilde{Q}_m(t)$. More precisely, we have
\begin{equation}\label{eqn:show-it}
\left\|(-\widetilde{Q}_m(-t)) - \frac{1}{t+a}\right\|_{\infty} =
\inf_{p\in \mathcal{P}_m}\max_{t\in[-1,1]}\left|(-p(-t)) - \frac{1}{t+a}\right|=
\inf_{q\in \mathcal{P}_m}\left\|q - \frac{1}{t+a}\right\|_{\infty},
\end{equation}
which shows that the best polynomial approximation to
$\frac{1}{t+a}$ on $[-1,1]$ is $(-\widetilde{Q}_m(-t))$.

\subsection{Proof of Theorem \ref {thm:polynomial}\label{sect:thmproof}} 
We need to show that
for the polynomial $Q_m(t)$ defined as in~\eqref{eqn:bpa1x} we have 
$Q_m(t)=(-\widetilde{Q}_m(-t))$. 
We use properties of Chebyshev polynomials to
prove this identity. 
If we set  $\alpha=\arccos t$ we have
\begin{eqnarray*}
(t^2-1) T^{\prime}_m(-t) &=&  (-1)^{m}m\sin\alpha\sin m\alpha
=\frac{(-1)^{m-1}m}{2}(\cos(m+1)\alpha-\cos(m-1)\alpha)\\
& = &\frac{(-1)^{m-1}m}{2}(T_{m+1}(t)-T_{m-1}(t)).
\end{eqnarray*}
Recall that $T_k(-t)=(-1)^kT_k(t)$, $\delta = (a-\sqrt{a^2-1})$,
and $2t T_m(t)
=(T_{m+1}(t)+T_{m-1}(t))$. Therefore, we have
\begin{eqnarray*}
\widetilde{R}_{m+1}(-t) &=& 
-(a t+1)T_{m}(-t)+\frac{\sqrt{a^2-1}}{m}(t^2-1)
T_{m}^\prime(-t) \\
&=& 
(-1)^{m+1}(a t+1)T_{m}(t)+\frac{(-1)^{m-1}\sqrt{a^2-1}}{2}(T_{m+1}(t)-T_{m-1}(t))\\
&=& \frac{(-1)^{m+1}}{2}
\left[(a (T_{m+1}(t)+T_{m-1}(t))
+2T_{m}(t)+\sqrt{a^2-1} (T_{m+1}(t)-T_{m-1}(t)) \right]\\
&=& \frac{(-1)^{m+1}}{2}
\left[\delta^{-1} T_{m+1}(t)+2 T_{m}(t)+\delta T_{m-1}(t) \right]\\
&=& \frac{(-1)^{m}}{2}
\left[\eta^{-1} T_{m+1}(t)-2 T_{m}(t)+\eta T_{m-1}(t) \right].
\end{eqnarray*}
Looking at the definition of $R_{m+1}(t)$, given in Theorem~\ref{thm:polynomial} (relation \eqref{eqn:rmplus1})
it is easily seen that  $\widetilde{R}_{m+1}(-t) =
\frac{(-1)^m}{2}R_{m+1}(t)$. Since $(\eta^{-1}-\eta)^2=4(a^2-1)$,
we finally get
\begin{eqnarray*}
(-\widetilde{Q}_m(-t)) &=& 
-\frac{1}{-t-a}\left(1 -  \frac{4\delta^{m}}{(\delta+\delta^{-1})^2}
\widetilde{R}_{m+1}(-t)\right)= \frac{1}{t+a}\left(1 -  \frac{2(-1)^m\delta^{m}}{(\delta^{-1}-\delta)^2}
R_{m+1}(t)\right) \\
& = &  \frac{1}{t+a}\left(1 -  \frac{2\eta^{m}}{(\eta-\eta^{-1})^2}
R_{m+1}(t)\right) = Q_{m+1}(t).
\end{eqnarray*}
Thus, $Q_m(t)$ and $(-\widetilde{Q}_m(-t))$ coincide and the proof is complete.\hfill $\Box$
 \begin{remark}
   It is also possible to prove directly that the polynomial
   in~\eqref{eqn:bpa1x} is a polynomial of best approximation to
   $x^{-1}$ by specifying the points of Chebyshev alternance. Such
   proof is however much more elaborate than the one presented here.
 \end{remark}

\end{document}